\documentclass[a4paper,12pt]{article}

\usepackage{amsmath,amsthm,amssymb,amscd}
\usepackage{pgfplots}
\usepackage{amssymb}
\usepackage{amsthm}
\usepackage{amsmath, amsfonts}
\usepackage{authblk}
\usepackage{epsfig}
\usepackage{amscd}
\usepackage{graphicx,wrapfig}
\usepackage{subcaption}
\usepackage[thinlines,thiklines]{easybmat}
\usepackage{nicematrix}
\usepackage{hyperref}
\usepackage{capt-of}
\usepackage{float}
\usepackage{stackrel}
\usepackage{tikz-cd}
\usepackage{xcolor}
\usepackage{fancybox}
\usepackage{float}
\def\natu           {\mathbb N}

\def\real		{\mathbb R}

\def\comp		{\mathbb C}

\def\F		{\mathbb F}

\def\R		{{\cal R}}
\def\L		{{\cal L}}
\def\M		{{\cal M}}

\def\c		{{\rm circ}}
\def\O		{{\cal O}}

\title{Periodicities in the Riordan arrays of polynomials over finite fields}

\author{Derek E. Bellamy, Eva N. Pflomm, Nikolai A. Krylov 
\thanks{Nikolai Krylov is supported by an AMS-Simons 
Research Enhancement Grant for Primarily Undergraduate 
Institution Faculty.}}
\affil{Siena University\\ Department of Mathematics and Statistics\\
515 Loudon Road, Loudonville NY 12211, USA\\
de26bell@siena.edu, en24pflo@siena.edu, nkrylov@siena.edu}

\date{}

\begin{document}

\newtheorem{theorem}{Theorem}
\newtheorem{lemma}[theorem]{Lemma}
\newtheorem{claim}[theorem]{Claim}
\newtheorem{corollary}[theorem]{Corollary}
\newtheorem{conj}[theorem]{Conjecture}
\newtheorem{prop}[theorem]{Proposition}
\newtheorem{note}{Note}
\theoremstyle{definition}
\newtheorem{definition}[theorem]{Definition}
\newtheorem{example}[theorem]{Example}
\numberwithin{equation}{section}

\textheight=22cm
\topmargin=-4mm

\maketitle

\begin{abstract}
We study periodicity properties of the 2-D $\bigl(p_1(t)/p_2(t),\, tp_3(t)\bigr)$ and 3-D 
$\bigl(p_1(t)/p_2(t),\, tp_3(t),\, p_4(t)\bigr)$ Riordan arrays over a finite field $\F_q$, 
where each $p_i(t)$ is a polynomial with $p_i(0)\neq 0$. We show that the columns of 
the 2-D Riordan array are eventually periodic sequences, where a circulant matrix generated 
by the coefficients of $p_3(t)$ determines the behavior of this periodicity as the 
column index grows indefinitely. Furthermore, we prove that the preperiodic 
column partial sums of the 2-D array are periodic, and present a family of the 
Riordan arrays for which such sequences of partial sums are identically zero. 
We also show that the layers of the 3-D Riordan array contain periodic orbits 
related to each other via powers of a circulant matrix generated by 
the coefficients of $p_4(t)$. 
\end{abstract}


\tableofcontents

%%%%%%%%%%%%%%%%%%%%%%%%%%%
%%%%%%%%%%%%%%%%%%%%%%%%%%%
%%%%%%%%                          %%%%%%%%%%%
%%%%%%%%   Introduction    %%%%%%%%%%%
%%%%%%%%                          %%%%%%%%%%%
%%%%%%%%%%%%%%%%%%%%%%%%%%%
%%%%%%%%%%%%%%%%%%%%%%%%%%%

\section{Introduction}

Consider the rational function $1/(1-2t)$ expanded into a 
formal power series over the reals
$$
\frac{1}{1-2t} = \sum\limits_{i=0}^{\infty} 2^it^i \in\real[[t]].
$$ 
Coefficients of this series grow unboundedly. However, over a finite field $\F_p$, 
where $p$ is an odd prime, coefficients of the same formal power series form a 
periodic sequence, since $2^{p-1}\equiv 1\pmod{p}$.
Over $\F_2$, this sequence $\{1,\,0,\,\ldots\}$ is {\sl eventually } (or {\sl ultimately}) 
periodic. It follows from polynomial long division and Theorem 8.40 of Lidle and 
Niederreiter \cite{FF}, or Theorem 1.5 of Everest et al. \cite{RS}, 
that over a finite field, coefficients of the formal power series 
expansion of any rational function satisfy a linear recurrence relation, 
and therefore form an ultimately periodic sequence. Hence, if we take 
two rational functions $g(t)$ and $f(t)$ such that $g(0)\neq 0$ and $f(0)=0$, 
and consider the Riordan array 
$\bigl(g(t),\,f(t)\bigr)$ over a finite field, its $k$-th column ($k\geq 0$) will be 
eventually periodic, since its entries are the coefficients of the formal power 
series $g(t)f^k(t)$. 

When both $g(t)$ and $f(t)$ are polynomials, each column will have only 
finitely many non-zero terms, and thus, a trivial ultimate periodicity. 
In this paper, we are interested in periodic properties of the Riordan arrays 
$\bigl(g(t),\,f(t)\bigr)$ and $\bigl(g(t),\, f(t),\, h(t)\bigr)$, where $f(t)$ and $h(t)$ 
are polynomials, and $g(t)$ is a non-polynomial rational function over a finite field. 

The structure of the article is the following. A brief overview of Riordan arrays and linear 
recurring sequences is given next. In particular, we show where the periodicity in the 
coefficients of a formal power series begins. Then, we explain 
how to obtain the periodic blocks of the columns of the Riordan array 
$\bigl(g(t),\,f(t)\bigr)$ using the periodic block of $g(t)$ and powers of 
a circulant matrix generated by the coefficients of $f(t)$. 
(see Theorem \ref{MainThm1}). At the end of that section, 
we provide an example proving that the corresponding $A$ and $Z$ 
sequences of the Riordan array $\bigl(g(t),\,f(t)\bigr)$ need not 
be periodic in general. In section 4, we consider three-dimensional 
Riordan array $\bigl(g(t),\,f(t),\,h(t)\bigr)$ of polynomials 
and determine how the periodic orbits in the layers are related to 
each other via a circulant matrix generated by the coefficients of $h(t)$
(Theorem \ref{MainThm2}). 
In the last section we prove that the preperiodic partial sums in the 
columns of the Riordan array $\bigl(g(t),\,f(t)\bigr)$ are periodic 
(see Theorem \ref{PPSThm}) and present  
an infinite family of particular Riordan arrays, where the sequence of such 
partial sums is identically zero. 

~

\noindent {\bf Acknowledgement}

~

This research was funded by the Center for Undergraduate 
Research and Creative Activity (CURCA) at Siena University under the Summer 
Scholars Program. The authors also acknowledge the use of AI tools for 
scholarly literature exploration, proof verification, and manuscript preparation.

%%%%%%%%%%%%%%%%%%%%%%%%%%%
%%%%%%%%%%%%%%%%%%%%%%%%%%%
%%%%%%%%                      %%%%%%%%%%%%
%%%%%%%%   Section 2    %%%%%%%%%%%%
%%%%%%%%                      %%%%%%%%%%%%
%%%%%%%%%%%%%%%%%%%%%%%%%%%
%%%%%%%%%%%%%%%%%%%%%%%%%%%

\section{Preliminaries}

\subsection{Riordan Arrays}

For the detailed introduction to the subject, the reader is referred to 
the books by Barry \cite{Barry} and Shapiro et al. \cite{Shapiro0}, 
and a survey article by Davenport et al. \cite{Davenport}. Below are a few basic 
definitions and notations used throughout.

We fix an arbitrary finite field $\F_q$, where $q = p^m$, for some $m\in\natu$ and a 
rational prime $p\in\natu$. The set of all formal power series (f.p.s. from now on) in 
an indeterminate $t$ with coefficients in $\F_q$ will be denoted by $\F_q[[t]]$.
The \emph{order} of $f(t) =\sum_{k=0}^\infty f_kt^k \in \F_q[[t]]$, 
is the minimal number $r\in{\mathbb N_0=\{0\}\cup\natu}$ such that $f_r \neq 0$, 
and the set of all f.p.s. of order $r$ is denoted by ${\cal F}_r$.

Let $g(t) \in {\cal F}_0$ and $f(t) \in {\cal F}_1$; 
the pair $\bigl(g(t) ,\,f(t)\bigr)$ defines the {\em proper Riordan array} 
$$
\R =(d_{n,k})_{n,k\geq 0}=\bigl(g(t) ,\,f(t)\bigr)
$$ 
having
\begin{equation}\label{1}
d_{n,k} = [t^n]g(t) f(t)^k,
\end{equation} 
where $[t^n]h(t)$ denotes the coefficient of $t^n$ in the expansion of a f.p.s. $h(t)$. 
The set of all such lower triangular matrices over $\F_q$ forms a multiplicative group, 
called the {\sl Riordan group}. The group operation here is the usual matrix multiplication, 
which is written in terms of the f.p.s. as 
\begin{equation}
\label{GRoper}
\bigl(g(t),\,f(t)\bigr) * \bigl(v(t),\,u(t)\bigr) = \bigl(g(t)v(f(t)),\,u(f(t))\bigr), 
\end{equation}
with the Riordan array $I = (1,\,t)$ acting as the group identity. 
The inverse of the Riordan array $\bigl(g(t) ,\,f(t)\bigr)$, is the pair
$$
\bigl(g(t) ,\,f(t)\bigr)^{-1} = \left(\frac{1}{g(\bar{f}(t))},\,\bar{f}(t)\right),
$$
where we used the standard notation $\bar{f}(t)$ for the compositional inverse 
of $f(t)$, that is $\bar{f}(f(t)) = t$ and $f(\bar{f}(t)) = t$. 

Periodicities in the Riordan array $\bigl(1/(1-t^{d+1}),\,tp(t)\bigr)$, where 
$p(t)\in\comp[t]$ is a polynomial of degree $d$, have been studied by 
the third author in \cite{Krylov1} and \cite{Krylov2}. Here we will discuss 
similar periodic properties in the Riordan arrays $\bigl(p_1(t)/p_2(t),\,tp_3(t)\bigr)$ and 
$\bigl(p_1(t)/p_2(t),\,tp_3(t),\,p_4(t)\bigr)$, where $p_i(t)\in \F_q[t]$.

%%%%%%%%                           %%%%%%%%%%%%
%%%%%%%%   Subsection     %%%%%%%%%%%%
%%%%%%%%                           %%%%%%%%%%%%

\subsection{Linear Recurring Sequences}

It is well known that a linear recurrence relation over a finite field 
generates an ultimately periodic sequence. It is also straightforward to 
see how the formal power series expansion of a rational function defines a 
linear recurrence relation. Here, we briefly recall the corresponding terminology 
and, for the sake of completeness and convenience, state and prove a few  
basic results, which will be used repeatedly in what follows. 
The reader will find comprehensive details in the monographs 
\cite[Chapter~8]{FF} and \cite[Chapter~3]{RS}. In particular, our 
Lemma~\ref{LPer} corresponds to the first part in Theorem~3.2 of \cite{RS}.

Take two polynomials over $\F_q$, 
$p_1(t)$ and $p_2(t)$ of degrees $d_1$ and $d_2$ respectively, and 
assume that they are relatively prime. 
Since the ratio $p_1(t)/p_2(t)$ will be the generating function of 
the zeroth column of a Riordan array, we require that $p_i(0)\neq 0$. 
This ratio determines a unique {\sl $d_2$-th order homogeneous 
linear recurring sequence} (abbreviated as L.R.S. from now on) in 
$\F_q$. Indeed, since the ring $\F_q[t]$ is a Euclidean domain, there 
exist unique polynomials $q(t)$ and $r(t)$ such that 
\begin{equation}
\label{QR1}
p_1(t) = q(t)p_2(t) + r(t),
\end{equation} 
and $0\leq \deg(r(t)) < \deg(p_2(t))$. If $d_1\geq d_2$, consider the 
ratio $r(t)/p_2(t)$ instead of $p_1(t)/p_2(t)$, and since  
$\gcd(r(t),\,p_2(t))  = \gcd(p_1(t),\,p_2(t)) =1$, we can assume here 
that $d_1 < d_2$. Write the polynomials $p_1(t)$ and $p_2(t)$ as 
$$
p_1(t) = a_0 + a_1 t + \dots + a_{d_1} t^{d_1}, ~ 
p_2(t) = b_0 + b_1 t + \dots + b_{d_2} t^{d_2}
$$ 
with $a_0a_{d_1}\neq 0$ and $b_0b_{d_2}\neq 0$, and the ratio $p_1(t)/p_2(t)$ 
as a f.p.s.
\begin{equation}
\label{FPS1}
\frac{p_1(t)}{p_2(t)} = \sum _{k=0}^{\infty }s_{k}t^{k}.
\end{equation}
Identity (\ref{FPS1}) is equivalent to 
$$
\left(b_0 + b_{1}t + \dots + b_{d_2}t^{d_2}\right)
\left(s_0 + s_1t + \dots \right) = a_0 + a_{1}t + \dots + a_{d_1}t^{d_1}.
$$
Equating the coefficients of $t^k$ on both sides 
yields the initial conditions 
$$
b_0s_0 = a_0, ~  b_0s_1 + b_1s_0 = a_1, ~  
\ldots , ~ \sum\limits_{i=0}^{d_1} b_is_{d_1-i} = a_{d_1},
$$ 
and for each $k\geq d_1 + 1$ 
$$
\sum\limits_{i=0}^{k} b_is_{k - i} = 0,
$$
where we assumed $b_i = 0$ when $i > d_2$. In particular,  
for every integer $n\geq 0$ taking $k = d_2 + n$ one obtains the relation
\begin{equation}
\label{LRR1}
s_{d_2 + n} = s_{d_2 + n - 1}\left(\frac{-b_1}{b_0}\right) + 
s_{d_2 + n-2}\left(\frac{-b_2}{b_0}\right) + 
\cdots + s_n \left(\frac{-b_{d_2}}{b_0}\right),
\end{equation}
which is equivalent to the relation (8.1) from Chapter 8 of \cite{FF}. In this way  
the ratio $p_1(t)/p_2(t)$ determines a $d_2$-th order homogeneous L.R.S. 
$\{s_i\}_{i\geq 0}$ in $\F_q$. When $d_1 \geq d_2$, we add the polynomial $q(t)$ from 
(\ref{QR1}) to the generating function $r(t)/p_2(t)$ and obtain the generating 
function for $p_1(t)/p_2(t)$. The corresponding sequence satisfies the same 
recurrence relation as before, but with shifted 
indices and different initial conditions. The polynomial 
\begin{equation}
\label{CharPol1}
f(t) = t^{d_2} + \left(\frac{b_1}{b_0}\right) t^{d_2 - 1} + 
\left(\frac{b_2}{b_0}\right) t^{d_2 - 2} + \cdots +  
\left(\frac{b_{d_2-1}}{b_0}\right) t  + \left(\frac{b_{d_2}}{b_0}\right)
\end{equation}
equals the characteristic polynomial  $\chi(t) = \det(tI-A)$ of the associated matrix
\begin{equation}
\label{AsMA}
A = 
\begin{pmatrix}
0 & 0 & 0 & \cdots & 0 & -b_{d_2}/b_0\\
1 & 0 & 0 & \cdots & 0 & -b_{d_2-1}/b_0\\
0 & 1 & 0 & \cdots & 0 & -b_{d_2-2}/b_0\\
\vdots & \vdots & \vdots & & \vdots & \vdots\\
0 & 0 & 0 & \cdots & 1 & -b_{1}/b_0\\
\end{pmatrix}\in \mbox{GL}(d,\,\F_q),
\end{equation}
and is called the {\sl characteristic polynomial} of the L.R.S. $\{s_i\}_{i\geq 0}$ 
determined by (\ref{LRR1}). Since $b_{d_2}/b_0\neq 0$, Theorems 8.7 and 
8.11 from \cite{FF} imply that every L.R.S. satisfying (\ref{LRR1}) will be 
periodic with least period $\pi \leq q^{d_2} -1$. The sequence $\{s_i\}_{i\geq 0}$ 
is called {\sl periodic} (or {\sl purely periodic}) with a period $\pi$ if 
$s_{j+\pi} = s_j$ for all $j\geq 0$. If the periodicity begins at a later term, 
we call such a sequence {\sl ultimately or eventually periodic}.

The same L.R.S. satisfies many linear recurrence relations, 
and the least period of the sequence depends largely on the characteristic 
polynomial $f(t)$ of the sequence. For example, Theorem 8.27 from \cite{FF} 
states that the least period of the sequence divides the order of $f(x)$, 
which is the least positive integer $e$ for which $f(t)$ divides $t^e - 1$ in 
$\F_q[t]$. We refer the reader to \cite[Chapter~8]{FF} and 
\cite[Chapter~3]{RS} for the corresponding proofs and further details 
regarding the properties of the least period of a L.R.S. 

We want to show that the linear recurring sequences corresponding to the 
fractions $p_1(t)/p_2(t)$ and $1/p_2(t)$ with $\gcd(p_1(t),\,p_2(t)) =1$ 
have the same least period  (cf. Theorem~3.2 from \cite{RS}). 
To do so, we apply Theorem 8.44 from \cite{FF}, which states that the 
least period of a L.R.S. satisfying (\ref{LRR1}) equals the order of its 
minimal polynomial. For a L.R.S. $\{s_i\}_{i\geq 0}$ 
there exists a unique monic polynomial $m(t)\in \F_q[t]$ with the following 
property: a monic polynomial $f(t)\in \F_q[t]$ of positive degree is a 
characteristic polynomial of $\{s_i\}_{i\geq 0}$ if and only if $m(t)$ 
divides $f(t)$. This unique polynomial $m(t)$ is called the {\sl minimal 
polynomial} of the sequence (see \cite[\S 8.4]{FF}). 

\begin{lemma}
\label{LPer}
If $\gcd(p_1(t),\,p_2(t)) =1$, then the linear recurring sequences determined 
by fractions $p_1(t)/p_2(t)$ and $1/p_2(t)$ have equal least periods.
\end{lemma}
\begin{proof}
Since Theorem 8.44 from \cite{FF} states that the least period of a L.R.S. 
equals the order of the minimal polynomial, it is enough to prove that the 
sequences corresponding to two given fractions have the same 
minimal polynomials. In fact, we show that both minimal polynomials equal  
the polynomial $(1/b_0)t^{d_2}p_2(1/t)$. Write the 
ratio $p_1(t)/p_2(t)$ as a generating function 
\begin{equation}
\label{Frac1}
\frac{p_1(t)}{p_2(t)} = S(t) = \sum\limits_{i=0}^{\infty} s_it^i.
\end{equation}
As we explained above, the sequence of coefficients $\{s_i\}_{i\geq 0}$ is 
ultimately periodic with the characteristic polynomial $f(t)$ given 
in (\ref{CharPol1}). Notice that 
$$
b_0t^{d_2}f(1/t) = b_0 + b_1t + \cdots + b_{d_2}t^{d_2} = p_2(t),
$$
that is, $p_2(t)$ is a scalar multiple of the {\sl reciprocal characteristic 
polynomial} $f^*(t) = t^{d_2}f(1/t)$. According to \cite[Theorem 8.40]{FF}, 
the generating function $S(t)$ equals 
\begin{equation}
\label{GF1}
\frac{g(t)}{f^*(t)}
\end{equation}
for some polynomial $g(t)$, satisfying $\deg(g(t)) < \deg(f^*(t))$, and since 
$f^*(t) = t^{d_2}f(1/t) = p_2(t)/b_0$ we can rewrite (\ref{GF1}) as
\begin{equation}
\label{GF2}
S(t) = \frac{g(t)}{f^*(t)} = \frac{b_0g(t)}{p_2(t)} =  \frac{p_1(t)}{p_2(t)}.
\end{equation}
Condition $\gcd(p_1(t),\,p_2(t)) =1$ implies that $b_0g(t) = p_1(t)$. Furthermore, 
if we denote the minimal polynomial of the L.R.S. $\{s_i\}_{i\geq 0}$ by 
$m(t)$, then Theorem 8.42 from \cite{FF} implies that $m(t)$ will also be a characteristic 
polynomial of the sequence $\{s_i\}_{i\geq 0}$, and $f(t) = m(t)h(t)$ for some 
$h(t)\in\F_q[t]$. Using Theorem 8.40 again, there is a polynomial $g_2(t)$ 
such that
\begin{equation}
\label{GF3}
\frac{p_1(t)}{p_2(t)} = S(t) = \frac{g_2(t)}{m^*(t)} = \frac{g(t)}{f^*(t)} 
=  \frac{g(t)}{m^*(t)h^*(t)}.
\end{equation}
The last equality in (\ref{GF3}) implies $g(t)=g_2(t)h^*(t)$, so  
$p_1(t) = b_0g_2(t)h^*(t)$. Since 
$$
p_2(t) = b_0f^*(t) = b_0m^*(t)h^*(t) ~ \mbox{and} ~ \gcd(p_1(t),\,p_2(t)) =1,
$$
the polynomial $h^*(t)$ is constant, and therefore $h(t)$ must also be 
constant. Since both $f(t)$ and $m(t)$ are monic, $h(t) = 1$ and $f(t) = m(t)$. 
In particular, $m(t) = (1/b_0)t^{d_2}p_2(1/t)$. Since the only condition 
needed from the polynomial $p_1(t)$ was the nonexistence of common 
factors with $p_2(t)$, exactly the same argument shows that the fraction 
$1/p_2(t)$ has the same minimal polynomial, and the proof is finished.
\end{proof}

\begin{note}
According to section \S 8.2 of \cite{FF}, if the characteristic 
polynomial $f(t)$ of a L.R.S. $\{s_i\}_{i\geq 0}$ is irreducible over $\F_q$ with 
$\deg(f(t))= d$, then the least period of this sequence 
divides $q^d - 1$.  Example 8.31 from \S 8.2 shows 
that this statement is false for the reducible polynomial 
$$
f(t) = 1+ t + t^5 = (1 + t + t^2)(1 + t^2 + t^3)\in \F_2[t].
$$ 
In this case $f^*(t) = 1+t^4+t^5$, and the periodic block of $1/f^*(t)$ 
$$
\{s_i\}_{i = 0}^{i = 20} = 
\{1,\, 0,\, 0,\, 0,\, 1,\, 1,\, 0,\, 0,\, 1,\, 0,\, 1,\, 0,\, 1,\, 1,\, 1,\, 1,\, 1,\, 0,\, 0,\, 0,\, 0\}
$$
has length 21, which does not divide $2^5 - 1 = 31$.
\end{note} 

Since every L.R.S. $\{s_i\}_{i\geq 0}\in \F_q$ is ultimately 
periodic, say with the least period $\pi$, we need 
to know at what term $r$ the periodicity $s_{r} = s_{r+\pi}$ begins. We 
answer this question in the next lemma. As before, we assume that 
$\gcd(p_1(t),\,p_2(t)) =1$ and $d_i = \deg(p_i(t))$.

\begin{lemma}
\label{FPSperiod}
Periodicity of the L.R.S. determined by the fraction $p_1(t)/p_2(t)$ begins at the 
$(d_1 - d_2 + 1)$-st term when $d_1 - d_2 + 1 > 0$, and at the zeroth term otherwise.
\end{lemma}
\begin{proof}
Use the division in $\F_q[t]$ and write  
$$
p_1(t) = q(t)p_2(t) + r(t), ~ \mbox{and} ~ d_1= \deg(q(t)) + d_2,
$$
with $r(t)/p_2(t) = \sum\limits_{i\geq 0} \rho_it^i$. Theorem 8.11 from \cite{FF} 
states that the L.R.S. $\{\rho_i\}_{i\geq 0}$ is periodic, and if we denote the 
least period by $\pi$, it means $\rho_k = \rho_{k+\pi}$ for all $k\geq 0$.
If $q(t) = q_0 + q_1t + \cdots + q_nt^n$ with $q_n\neq 0$ in $\F_q$, then 
$$
[t^n]\bigl(p_1(t)/p_2(t)\bigr) = q_n + \rho_n \neq \rho_{n+\pi} = 
[t^{n + \pi}]\bigl(p_1(t)/p_2(t)\bigr).
$$ 
However, for all $i\geq n+1$, we do have 
$$
[t^i]\bigl(p_1(t)/p_2(t)\bigr) = \rho_i = \rho_{i+\pi} = [t^{i+\pi}]\bigl(p_1(t)/p_2(t)\bigr).
$$
Thus, the sequence of coefficients of $p_1(t)/p_2(t)$ is ultimately 
periodic, and its periodicity starts at the $(n+1)$-st term. Since 
$$
n + 1 = \deg(q(t)) + 1 = d_1 - d_2 + 1,
$$
we proved the first statement when $d_1 - d_2 + 1 > 0$. When  
$d_1 - d_2 + 1 \leq 0$, the quotient polynomial $q(t) = 0$, 
and we have the periodic sequence determined by $r(t)/p_2(t)$, 
with the periodicity starting at the zeroth term.
\end{proof}

In the last lemma of this section, we prove that if the degree of the 
denominator $d_2 > d_1 + 1$, then the periodic block of the L.R.S. 
determined by $p_1(t)/p_2(t)$ ends with several zero terms 
(cf. the first few columns in the  Riordan array \ref{HardPer0} 
from Example \ref{Example5} below).

\begin{lemma}
\label{End0}
Let $p_1(t)$ and $p_2(t)$ be polynomials over $\F_q$ satisfying the same assumptions as in 
two previous lemmas and $d_2 > d_1 + 1$. Then, the last $d_2 - (d_1 + 1)$ terms of the 
periodic block of the L.R.S. determined by $p_1(t)/p_2(t)$ equal zero.
\end{lemma}
\begin{proof}
Using equality (\ref{Frac1}) with $p_1(t)=\sum\limits_{i=0}^{d_1}a_it^i$ and 
$p_2(t)=\sum\limits_{i=0}^{d_2}b_it^i$, we have
\begin{equation}
\label{Frac2}
a_0 + a_1t + \cdots + a_{d_1}t^{d_1} = 
\left(b_0 + b_1t + \cdots + b_{d_2}t^{d_2}\right)\sum\limits_{i = 0}^{\infty} s_it^i.
\end{equation}
Since the L.R.S. is periodic with period $\pi$, 
let us denote the sum $s_0 + s_1t + \cdots + s_{\pi-1}t^{\pi - 1}$ 
corresponding to its periodic block by $B(t)$. Then 
$$
\sum\limits_{i = 0}^{\infty} s_it^i = B(t) + t^{\pi}B(t) + \cdots = \frac{B(t)}{1 - t^{\pi}},
$$
and (\ref{Frac2}) implies the equality
\begin{equation}
\label{Frac3}
(1 - t^{\pi})\left(a_0 + a_1t + \cdots + a_{d_1}t^{d_1}\right) = 
\left(b_0 + b_1t + \cdots + b_{d_2}t^{d_2}\right)B(t).
\end{equation}
In particular, for the degrees we have 
$\pi + d_1 = d_2 + \deg(B(t))$, or $\deg(B(t)) = \pi + d_1 - d_2$. 
If we denote $d_2- (d_1 +1)$ by $D$, then $\deg(B(t)) = \pi - D - 1$, 
which makes $s_j = 0$ for each 
$$
j\in\{\pi  - D,\, \pi - D + 1,\, \ldots,\, \pi - 1\},
$$
and finishes the proof.
\end{proof}

%%%%%%%%%%%%%%%%%%%%%%%%%%%
%%%%%%%%%%%%%%%%%%%%%%%%%%%
%%%%%%%%                      %%%%%%%%%%%%
%%%%%%%%   Section 3    %%%%%%%%%%%%
%%%%%%%%                      %%%%%%%%%%%%
%%%%%%%%%%%%%%%%%%%%%%%%%%%
%%%%%%%%%%%%%%%%%%%%%%%%%%%

\section{Circulant Matrices and Periodicities}

Let us now take three polynomials $p_i(t)\in \F_q[t] ~ i\in\{1,2,3\}$, where $p_i(0)\neq 0$, 
$d_i = \deg(p_i(t))$, and $p_2(t)$ is relatively prime to the other two. 
The $k$-th column of the Riordan array 
\begin{equation}
\label{RAp123}
\left(\frac{p_1(t)}{p_2(t)},\,tp_3(t)\right)
\end{equation}
is given by the coefficients of the f.p.s. $\bigl(p_1(t)(tp_3(t))^k\bigr)/p_2(t)$. 
According to our previous section, the column is ultimately periodic with a fixed 
period, whose length is independent of $k$, and depends only on the 
polynomial $p_2(t)$. 
Periodicity in the $k$-th column will begin at the $(k(1+d_3)+d_1 - d_2 + 1)$-st
term by Lemma \ref{FPSperiod}, and to describe the periodic blocks in each 
column we use a circulant matrix generated by the coefficients of $p_3(t)$. 
We give below the definition of a circulant matrix generated by a vector, 
and refer the reader to monographs by Davis \cite{Davis} and Fuhrmann 
\cite{Fuhrmann} for detailed discussions of the circulant matrices and their properties. 

%%%%%%%%                           %%%%%%%%%%%%
%%%%%%%%   Subsection     %%%%%%%%%%%%
%%%%%%%%                           %%%%%%%%%%%%

\subsection{Circulant Matrices}

Start with a $(d + 1)$-vector 
$$
\vec{v} = (v_d, \, v_{d-1},\,  \ldots,\, v_0)\in \F_q^{d+1},
$$
and consider a shift operator $T:\F_q^{d+1}\to \F_q^{d+1}$ defined as 
$$
T(v_d, \, v_{d-1},\, \ldots, \, v_1,\, v_0):= (v_0, \, v_d,\, v_{d-1},\, \ldots, \, v_1).
$$
To the {\sl generating vector} $\vec{v}$, we associate the following 
$(d+1)\times (d+1)$ 
matrix $\c(\vec{v})$, called the {\sl circulant matrix}.
Its rows are given by iterations of the shift operator acting on the vector 
$\vec{v}$, i.e. the $i$-th row of $\c(\vec{v}) $ is 
$T^i \vec{v}$, for each $i\in\{0,\,\ldots,\, d \}$:
$$
\c(\vec{v})= 
\begin{pmatrix}
v_d & v_{d-1} & \cdots & v_1 & v_0\\
v_0 & v_{d} & \cdots & v_2 & v_1\\
\vdots & \vdots & \ddots & \vdots & \vdots\\
v_{d-2} & v_{d-3} & \cdots & v_d & v_{d-1}\\
v_{d-1} & v_{d-2} & \cdots & v_0 & v_d\\
\end{pmatrix}.
$$
For example, for $\vec{v} = (0,\, 1,\, 0,\, \ldots,\, 0) = \vec{e}_2\in\F_q^{d+1}$,
 we obtain the matrix 
\begin{equation}
\label{FundPermM}
P = \c(\vec{e}_2) = \begin{pmatrix}
0 & 1 & 0 & \cdots & 0\\
0 & 0 & 1 & \cdots & 0\\
\vdots & \vdots & \vdots & \ddots & \vdots\\
1 & 0 & 0 & \cdots & 0\\
\end{pmatrix},
\end{equation}
which corresponds to the {\sl forward shift} permutation, that is, to the cycle 
$\sigma = (1,\,2,\,\ldots,\, d+1)$ generating the cyclic group 
$C_{d+1}$. In particular, $P^{d+1} = Id$, and 
\begin{equation}
\label{FUNDcirc}
\c(v_d, \, v_{d-1},\,  \ldots,\, v_0) = v_dI + v_{d-1}P + \cdots + v_1P^{d-1} + v_0P^d.
\end{equation}

%%%%%%%%                           %%%%%%%%%%%%
%%%%%%%%   Subsection     %%%%%%%%%%%%
%%%%%%%%                           %%%%%%%%%%%%

\subsection{Column Periodic Blocks}

The next theorem describes the periodic blocks in the columns of the Riordan 
array (\ref{RAp123}) when degrees of the denominator polynomial $p_2(t)$ 
and the numerator polynomial $p_1(t)$ satisfy the inequality $d_2\leq 1 + d_1$. 
We use this condition because otherwise the periodicity does not begin with 
the first column. Additional details are given in Example \ref{Example5} 
at the end of this subsection.

\begin{theorem}
\label{MainThm1}
Take $p_i(t)\in \F_q[t] ~ i\in\{1,2,3\}$, where $p_i(0)\neq 0$, and $p_2(t)$ 
is relatively prime to $p_1(t)$ and $p_3(t)$. Let $d_i = \deg(p_i(t))$ with $d_2\leq d_1 + 1$, 
and $\pi$ be the least period of the coefficient 
sequence of the f.p.s. $1/p_2(t)$. Further, let $P = \c(\vec{e}_2)$ as in 
(\ref{FundPermM}), $p_3(t) = c_0 + c_1t + \cdots + c_{d_3}t^{d_3}$, 
and $A=\c(\vec{v}) \in GL(\pi,\,\F_q)$ be the circulant $\pi\times\pi$ 
matrix generated by the vector
\begin{equation}
\label{Vec1}
\vec{v} = \Bigl( \sum\limits_{i=0}^{\lfloor\frac{d_3}{\pi}\rfloor} c_{d_3 - \pi i},\, 
\sum\limits_{i=0}^{\lfloor\frac{d_3 - 1}{\pi}\rfloor} c_{d_3 - 1 - \pi i},\,\ldots ,\, 
\sum\limits_{i=0}^{\lfloor\frac{d_3 -  (\pi -1)}{\pi}\rfloor} 
c_{d_3 - (\pi - 1) - \pi i} \Bigr) 
\end{equation}
if $\pi < d_3 + 1$, or by the vector
\begin{equation}
\label{Vec2}
\vec{v} = (c_{d_3},\, c_{d_3 - 1} ,\, \ldots, \, c_1 ,\, c_0,\, 0,\,\ldots,\, 0) ~ 
\mbox{if} ~ \pi \geq  d_3 + 1.
\end{equation}
Then
\begin{itemize}
\item[]{(i)} for each $k\geq 1$ the sequence of coefficients of 
the $k$-th column of the Riordan array (\ref{RAp123}) is ultimately periodic with the least period 
$\pi$. Periodicity starts at the $\bigl((d_3+1)k + d_1 - d_2  + 1 \bigr)$-st term when 
$(d_3+1)k + d_1 +1 > d_2$, or at the zeroth term otherwise.
\item[]{(ii)} for each $k\geq 1$ the periodic block in the $k$-th column equals
$$
C_k = A^k \cdot C_0,
$$
where $C_0$ is the periodic block in the $0$th column.
\end{itemize}
\end{theorem}
%%%%%%%%   Proof of the main theorem    %%%%%%%%%%%%
\begin{proof}
Statement (i) follows immediately from Lemmas \ref{LPer} and \ref{FPSperiod}.
Using notations from (\ref{Frac1}) of Lemma \ref{LPer}, we can say that the 
first periodic block in the zeroth column of (\ref{RAp123}) will be 
$C_0 = (s_r,\,s_{r+1},\,\ldots,\, s_{r+ \pi-1})^T$, where by 
Lemma (\ref{FPSperiod}), $r = d_1 + 1 - d_2$. 
To prove (ii) for $k=1$, we need to show that the periodic block of the first column 
$C_1$ equals $A\cdot (s_r,\,s_{r+1},\,\ldots,\, s_{r+ \pi-1})^T$. To describe 
this block we will use the coefficients 
extractor operator $[t^n]f(t)$. For the details and main properties of 
this operator we refer the reader to the books by Barry \cite{Barry}, 
\S 4.5.5, and Shapiro, et al. \cite{Shapiro0}, \S 2.2. It follows from (i) 
that the first periodic block in the first column of (\ref{RAp123}) equals 
$$
C_1 = \begin{pmatrix}
[t^{(d_3+1+d_1) - (d_2-1)}]\left(\frac{p_1(t)tp_3(t)}{p_2(t)}\right)\\
[t^{(d_3+1+d_1) - (d_2-1)+1}]\left(\frac{p_1(t)tp_3(t)}{p_2(t)}\right)\\ \vdots \\ 
[t^{(d_3+1+d_1) - (d_2-1)+\pi-1}]\left(\frac{p_1(t)tp_3(t)}{p_2(t)}\right)
\end{pmatrix}.
$$
Writing $\bigl(p_1(t)tp_3(t)\bigr)/p_2(t)$ as the product 
\begin{equation}
\label{Conv1}
\frac{p_1(t)tp_3(t)}{p_2(t)} = \bigl(t(c_0 + c_1t + \ldots + 
c_{d_3}t^{d_3})\bigr)\sum\limits_{i=0}^{\infty} s_it^i,
\end{equation}
and applying the shifting and convolution of $[t^n]$, 
we obtain for every $m\geq 0$ 
$$
[t^{(d_3+1+d_1) - (d_2-1) + m}]\left(\frac{p_1(t)tp_3(t)}{p_2(t)}\right) = 
[t^{(d_3+1+d_1) + m - d_2 }]\left(\frac{p_1(t)p_3(t)}{p_2(t)}\right) 
$$
$$
= [t^{(d_3+1+d_1) + m - d_2 }]\Bigl((c_0 + c_1t + \ldots + 
c_{d_3}t^{d_3})\sum\limits_{i=0}^{\infty} s_it^i\Bigr) 
$$
$$
= c_{d_3}s_{m - d_2 + d_1 + 1} + c_{d_3-1}s_{m - d_2 + d_1 + 2} + \\
\cdots + c_0s_{m - d_2 + d_3 + d_1 + 1} 
$$
\begin{equation}
\label{Conv2}
=\sum\limits_{i=0}^{d_3} c_is_{m - d_2 + d_3 + d_1+ 1 - i} 
=\sum\limits_{i=0}^{d_3} c_is_{m + r + d_3 - i}.
\end{equation}
For the last equality in (\ref{Conv2}), recall that $r = 1+d_1-d_2\geq 0$. 
Consider next two separate cases.

~

\noindent \underline{{\rm Case 1.} $\pi \geq d_3 + 1$}:\\
For $m = 0$, we can write this sum as the dot product 
\begin{equation}
\label{DotP}
(c_{d_3},\, c_{d_3 - 1},\, \ldots,\, c_0,\, 0,\,\ldots,\, 0)\cdot 
\begin{pmatrix}
s_{r}\\ 
s_{r + 1}\\
\vdots\\
s_{r + d_3}\\
s_{r + d_3 + 1}\\
\vdots\\
s_{r + \pi - 1}
\end{pmatrix},
\end{equation}
which shows that the first element of the periodic block $C_1$ equals $ \vec{v}\cdot C_0$.
The same argument with $m = 1$ in (\ref{Conv2}) will prove that the second element of the 
periodic block $C_1$ equals the dot product 
\begin{equation}
\label{DotP2}
(0,\, c_{d_3},\, c_{d_3 - 1},\, \ldots,\, c_0,\, 0,\,\ldots,\, 0)\cdot 
\begin{pmatrix}
s_{r}\\ 
s_{r + 1}\\
\vdots\\
s_{r + d_3}\\
s_{r + d_3 + 1}\\
\vdots\\
s_{r + \pi - 1}
\end{pmatrix} = T\vec{v}\cdot C_0,
\end{equation}
and so on. For the last choice, take $m = \pi - 1$. Since $s_r = s_{\pi + r}$, in this 
case the sum (\ref{Conv2}) can be written as the dot product
\begin{equation}
\label{DotP3}
(c_{d_3 - 1},\, c_{d_3 - 2},\, \ldots,\, c_0,\, 0,\,\ldots,\, 0,\, c_{d_3})\cdot 
\begin{pmatrix}
s_{r}\\ 
s_{r + 1}\\
\vdots\\
s_{r + d_3}\\
s_{r + d_3 + 1}\\
\vdots\\
s_{r + \pi - 1}
\end{pmatrix} = T^{\pi-1}\vec{v}\cdot C_0.
\end{equation}
It proves that the periodic block in the first column indeed equals $C_1 = A\cdot C_0$. 

Next, we want to use the induction, so let us use the same notation as 
before for the the f.p.s. in the $k$-th column, that is
$$
\frac{p_1(t)(tp_3(t))^k}{p_2(t)} = \sum\limits_{i\geq 0} s_it^i.
$$
We also denote the first periodic block in this $k$-th column by 
$$
C_k = (s_r,\,s_{r+1},\,\ldots,\, s_{r+ \pi-1})^T,
$$ 
where, according to Lemma (\ref{FPSperiod}), new $r = (d_3+1)k + d_1 + 1 - d_2 \geq 0$.
Then, we have the identity analogous to (\ref{Conv1})
\begin{equation}
\label{Conv3}
\frac{p_1(t)\bigl(tp_3(t)\bigr)^{k+1}}{p_2(t)} = \bigl(t(c_0 + c_1t + \ldots + 
c_{d_3}t^{d_3})\bigr)\sum\limits_{i=0}^{\infty} s_it^i,
\end{equation}
and similarly,  
$$
[t^{(d_3 + 1)(k+1) + d_1 - (d_2 - 1) + m}]\left(\frac{p_1(t)\bigl(tp_3(t)\bigr)^{k+1}}{p_2(t)}\right)
$$ 
\begin{equation}
\label{DotPr2}
=[t^{d_3 + m + r }]\left((c_0 + c_1t + \ldots + 
c_{d_3}t^{d_3})\sum\limits_{i=0}^{\infty} s_it^i\right) = \sum\limits_{i=0}^{d_3} c_is_{m + r + d_3 - i}.
\end{equation}
Using the same arguments as before, one shows that for $m = 0$, the last sum in (\ref{DotPr2}) equals 
the dot product $\vec{v}\cdot C_k$. Similarly, for any $m\in \{1,\,2,\,\ldots,\,\pi-1\}$ one proves that the 
$m$-th element of $C_{k+1}$ equals $T^{m-1}\vec{v}\cdot C_k$, and therefore,   
$C_{k+1} = A\cdot C_k = A^{k+1}\cdot C_0$.

~

\noindent \underline{{\rm Case 2.} $\pi < d_3 + 1$:}\\
Writing $d_3 = Q\pi + R$, where $Q\in\natu$ and $0\leq R < \pi$, 
and the sum (\ref{Conv2}) for $m = 0$ as 
\begin{equation}
\label{Conv2B}
c_{d_3}s_r + c_{d_3 -1}s_{r+1} + \cdots + c_1s_{r + d_3 -1} + c_0s_{r + d_3},
\end{equation}
we see that the number $s_r= s_{r+\pi} = s_{r+j\pi}$ appears in 
(\ref{Conv2B}) exactly $Q+1$ times, with the coefficients 
$\{c_{d_3},\,c_{d_3 - \pi},\,\ldots,\, c_{d_3 - Q\pi}\}$. Similarly, the number 
$s_{r+1}= s_{r + 1 + j\pi}$ appears 
in (\ref{Conv2B}) exactly $\lfloor(d_3 - 1)/\pi\rfloor + 1$ times, and so on. 
Hence, we can write the sum (\ref{Conv2B}) as the dot product
$$
\left(\sum\limits_{i = 0}^{\lfloor \frac{d_3}{\pi}\rfloor} c_{d_3 - i\pi} , \, 
\sum\limits_{i = 0}^{\lfloor \frac{d_3 - 1}{\pi}\rfloor} c_{d_3 - 1 - i\pi} , \, \ldots, \,
\sum\limits_{i = 0}^{\lfloor \frac{d_3 - (\pi - 1)}{\pi}\rfloor} c_{d_3 + 1 - \pi - i\pi} \right) \cdot 
\begin{pmatrix}
s_{r}\\ 
s_{r + 1}\\
\vdots\\
s_{r + d_3}\\
s_{r + d_3 + 1}\\
\vdots\\
s_{r + \pi - 1}
\end{pmatrix},
$$
which shows that the first element of the periodic block $C_1$ equals $ \vec{v}\cdot C_0$.
The rest of the proof is identical to the first case, and we leave the details to the reader.
\end{proof}

\begin{example}
\label{Example5}
Now, let us explain why the periodicity in column blocks depends on the 
restriction $d_2\leq d_1 + 1$. How exactly it depends will be investigated elsewhere. 
Consider the Riordan array over $\F_3$ 
\begin{equation}
\label{RAd1Ld2}
\left(\frac{1+t^2}{(1+ 2t)^9},\,t(1 + t)\right),
\end{equation}
where $(d_3+1)k+d_1+1-d_2 = 2k - 6 \leq 0$ for $k\in\{0,1,2,3\}$. The period here is 9 
and periodicity in columns 0 - 3 begins with the 0th term, in agreement with our 
lemmas from section 2. Here are the first twelve rows and columns with 
several periodic blocks boxed. 

\begin{equation}
\label{HardPer0}
\left(\frac{1+t^2}{(1+ 2t)^9},\,t(1 + t)\right) = 
\left(
\begin{BMAT}[3pt]{cccccccccccc}{cccccccccccc}
1 & 0 & 0 & 0 & 0 & 0 & 0 & 0 & 0 & 0 & 0 & 0\\
0 & 1 & 0 & 0 & 0 & 0 & 0 & 0 & 0 & 0 & 0 & 0\\
1 & 1 & 1 & 0 & 0 & 0 & 0 & 0 & 0 & 0 & 0 & 0\\
0 & 1 & 2 & 1 & 0 & 0 & 0 & 0 & 0 & 0 & 0 & 0\\
0 & 1 & 2 & 0 & 1 & 0 & 0 & 0 & 0 & 0 & 0 & 0\\
0 & 0 & 2 & 1 & 1 & 1 & 0 & 0 & 0 & 0 & 0 & 0\\
0 & 0 & 1 & 1 & 1 & 2 & 1 & 0 & 0 & 0 & 0 & 0\\
0 & 0 & 0 & 0 & 2 & 2 & 0 & 1 & 0 & 0 & 0 & 0\\
0 & 0 & 0 & 1 & 1 & 0 & 1 & 1 & 1 & 0 & 0 & 0\\
1 & 0 & 0 & 0 & 1 & 0 & 2 & 1 & 2 & 1 & 0 & 0\\
0 & 1 & 0 & 0 & 1 & 2 & 0 & 0 & 2 & 0 & 1 & 0\\
1 & 1 & 1 & 0 & 0 & 2 & 2 & 2 & 1 & 1 & 1 & 1
\addpath{(0,3,1)ruuuuuuuuulddddddddd}
\addpath{(1,3,1)ruuuuuuuuulddddddddd}
\addpath{(2,3,1)ruuuuuuuuulddddddddd}
\addpath{(3,3,1)ruuuuuuuuulddddddddd}
\addpath{(4,1,1)ruuuuuuuuulddddddddd}
\addpath{(5,0,1)uuuuuuuurdddddddd}
\addpath{(6,0,1)uuuuuurdddddd}
\end{BMAT}
\right)
\end{equation}
Also, according to Lemma \ref{End0}, the last 6 and 4 terms in the 
corresponding periodic blocks of the 0th and the 1st columns equal zero. Furthermore, 
the circulant matrix here is $A = \c(\vec{v})$, where $\vec{v} = (1,\,1,\,0,\,0,\,0,\,0,\,0,\,0,\,0)$, 
and comparing the periodic blocks $C_0$ and $C_1$ we see that
$$
\begin{pmatrix}
1& 1& 0& 0& 0& 0& 0& 0& 0\\
0& 1& 1& 0& 0& 0& 0& 0& 0\\
0& 0& 1& 1& 0& 0& 0& 0& 0\\
0& 0& 0& 1& 1& 0& 0& 0& 0\\
0& 0& 0& 0& 1& 1& 0& 0& 0\\
0& 0& 0& 0& 0& 1& 1& 0& 0\\
0& 0& 0& 0& 0& 0& 1& 1& 0\\
0& 0& 0& 0& 0& 0& 0& 1& 1\\
1& 0& 0& 0& 0& 0& 0& 0& 1
\end{pmatrix}\cdot
\begin{pmatrix}
1\\ 0\\1\\ 0\\ 0\\ 0\\ 0\\ 0\\ 0
\end{pmatrix} = \begin{pmatrix}
1\\ 1\\1\\ 0\\ 0\\ 0\\ 0\\ 0\\ 1
\end{pmatrix} \neq \begin{pmatrix}
0\\ 1\\1\\ 1\\ 1\\ 0\\ 0\\ 0\\ 0
\end{pmatrix} = P^7\cdot \begin{pmatrix}
1\\ 1\\1\\ 0\\ 0\\ 0\\ 0\\ 0\\ 1
\end{pmatrix},
$$
that is, $C_1 = P^7A\cdot C_0$, where $P = \c(\vec{e}_2)$, 
as in (\ref{FundPermM}). Similar computations show that
$$
C_2 = P^5A^2\cdot C_0 ~~ \mbox{and} ~~ C_3 = P^3 A^3\cdot C_0.
$$
Starting with the 3rd column, the 
power of the permutation matrix $P$ seems to become constant, since
$$
C_4 = P^3 A^4\cdot C_0 ~~ \mbox{and} ~~ C_5 = P^3 A^5\cdot C_0.
$$
Thus, we see in this example that 
$$
C_k = P^{9 - 2k} A^k\cdot C_0,~ \mbox{if}~ 0\leq k\leq 3,~\mbox{and} ~ 
C_k = P^3 A^k\cdot C_0 ~ \mbox{if}~ k \geq 4.
$$
Hence, in cases when $d_2 > d_1 + 1$, the second statement of 
Theorem \ref{MainThm1} requires a refinement, with   
matrix $P$ playing a significant role. We hope the precise description of such  
periodicity when $d_2 > d_1 +1$ will be given in future research.
\end{example}

%%%%%%%%                           %%%%%%%%%%%%
%%%%%%%%   Subsection     %%%%%%%%%%%%
%%%%%%%%                           %%%%%%%%%%%%

\subsection{No Periodicity in the $A$ and $Z$ Sequences}

Every Riordan array can be characterized by two sequences: 
the $A$ sequence, and the $Z$ sequence (see, for example \cite{Barry}, \S 5.3). 
In this short subsection, we give an example showing that, in general, neither the 
$A$ nor $Z$ sequences of the Riordan array (\ref{RAp123}) over 
a finite field is ultimately periodic. 

\begin{example}
Consider the following Riordan array over $\F_2$.
\begin{equation}
\label{NoPerA1}
 \left(\frac{1+t}{1+t+t^2},\,t(1+t)\right) = 
\left(
\begin{BMAT}[3pt]{cccccccccc}{cccccccccc}
1 & 0 & 0 & 0 & 0 & 0 & 0  & 0 & 0 & 0\\
0 & 1 & 0 & 0 & 0 & 0 & 0  & 0 & 0 & 0\\
1 & 1 & 1 & 0 & 0 & 0 & 0  & 0 & 0 & 0\\
1 & 1 & 0  & 1 & 0 & 0 & 0  & 0 & 0 & 0\\
0 & 0 & 0 & 1 & 1 & 0 & 0  & 0 & 0 & 0\\
1 & 1 & 1 & 0 & 0 & 1 & 0 & 0  & 0 & 0\\
1 & 1 & 1 & 1 & 1 & 1 & 1 & 0 & 0 & 0\\
0 & 0 & 0 & 0 & 1 & 1 & 0  & 1 & 0 & 0\\
1 & 1 & 1 & 1 & 1 & 0 & 0  & 1 & 1 & 0\\
1 & 1 & 1 & 1 & 1 & 0 & 1  & 0 & 0 & 1
\addpath{(0,7,1)ruuulddd}
\addpath{(1,5,1)ruuulddd}
\addpath{(2,3,1)ruuulddd}
\addpath{(3,1,1)ruuulddd}
\end{BMAT}
\right)
\end{equation}

Using Theorems 4.3 and 4.5 of \cite{Shapiro0}, we can write 
the generating functions of 
the Z and A-sequences for the Riordan array $(g,\,f)$ correspondingly as 
\begin{equation}
\label{A-seq}
Z(t) = \frac{g(\bar{f}(t)) - d_{00}}{\bar{f}(t) g(\bar{f}(t))} ~~ \mbox{and} ~~ A(t) = \frac{t}{\bar{f}(t)},
\end{equation}
where $d_{00}$ is the element of $(g,\,f)$ in the 0th row and 0th column. 
Since in this example $g(t) = (1+t)/(1+t+ t^2),~ f(t) = t(1+t)$, and $d_{00} = 1$, 
straightforward computations give 
\begin{equation}
\label{NoPerA20}
\bar{f}(t) = \frac{\sqrt{1 + 4t} - 1}{2} ~ \mbox{and} ~ g(\bar{f}(t)) =  \frac{1+\sqrt{1+4t}}{2(1+t)},
\end{equation}
and therefore 
\begin{equation}
\label{NoPerA2a}
Z(t) = \frac{\sqrt{1 + 4t} - 1- 2t}{2t} ~\mbox{and} ~ A(t) = \frac{2t}{\sqrt{1 + 4t} - 1}.
\end{equation}
For the $A$ sequence, we have 
\begin{equation}
\label{NoPerA2}
A(t) = \frac{2t}{\sqrt{1 + 4t} - 1}  = \frac{\sqrt{1 + 4t} + 1}{2} 
= 1 + \sum\limits_{n=1}^{\infty} (-1)^{n-1}C_{n-1}t^n,
\end{equation}
where $C_n=\frac{(2n)!}{(n+1)!n!}$ stands for the $n$-th Catalan number. It is well known 
(see, for example, the paper \cite{Alter} by Alter and Curtz or the paper 
\cite{Sagan} by Deutsch and Sagan) that a Catalan number 
$C_n$ is odd if and only if $n = 2^m - 1$ for some integer $m$. Hence, if we 
look at the f.p.s. (\ref{NoPerA2}) modulo 2, we will get the series
$$
A(t) \equiv 1+t+t^2 + t^4 + t^8 + t^{16} + t^{32} + \cdots \pmod{2},
$$
with coefficients, which are not ultimately periodic. Using results from \cite{Sagan}, one 
can show that the sequence of coefficients of the f.p.s. (\ref{NoPerA2}) is not 
ultimately periodic modulo 3 either. 

As for the $Z$ sequence of the Riordan array (\ref{NoPerA1}), comparing its formula in
(\ref{NoPerA2a}) with the fraction $(\sqrt{1+4t} + 1)/2$ we see that modulo 2
$$
Z(t) = \frac{A(t) -1 }{t} - 1 \equiv t + t^3  + t^7 + t^{15} + t^{31} + \cdots + t^{2^n-1} + \cdots,
$$
that is, the $Z$ sequence of (\ref{NoPerA1}) is also not ultimately periodic. 

\end{example}

We would like to close this section by noticing that the set of the Riordan arrays 
of type (\ref{RAp123}) is closed under the matrix multiplication, which is 
written in terms of the f.p.s. as (\ref{GRoper}). Since the product, as well as 
the composition of polynomials over a field $\F_q$, produce 
polynomials over $\F_q$ again, it is clear that the 
product of two Riordan arrays of type (\ref{RAp123}) 
will be another Riordan array of the same type. Thus, the question arises 
whether the set of all Riordan arrays of type (\ref{RAp123}) forms a 
group under the binary operation (\ref{GRoper}). The answer 
is negative, since the example (\ref{NoPerA1}) we just considered, 
shows that in general, the inverse of the Riordan array of type 
(\ref{RAp123}) will not be of the same type. Indeed, 
the inverse of $(g,\,f)$ equals $(1/g(\bar{f}),\, \bar{f})$, 
and using formulas (\ref{NoPerA20}), 
one finds that the inverse of the Riordan array (\ref{NoPerA1}) equals 
$$
\left(\frac{(1+t)(\sqrt{1+4t} - 1)}{2t},\,\frac{\sqrt{1+4t} - 1}{2}\right).
$$
Clearly, we cannot write $(\sqrt{1+4t} - 1)/2$ as a polynomial. Moreover, since 
$$
\frac{(1+t)(\sqrt{1+4t} - 1)}{2t} = 1 + t^2 + (t^3 + t^4) + (t^7 + t^8) +  \cdots + 
(t^{2^k-1} + t^{2^k}) + \cdots , 
$$
the sequence of its coefficients is not ultimately periodic, and cannot be generated 
by a ratio of polynomials.

%%%%%%%%%%%%%%%%%%%%%%%%%%%
%%%%%%%%%%%%%%%%%%%%%%%%%%%
%%%%%%%%                      %%%%%%%%%%%%
%%%%%%%%   Section 4    %%%%%%%%%%%%
%%%%%%%%                      %%%%%%%%%%%%
%%%%%%%%%%%%%%%%%%%%%%%%%%%
%%%%%%%%%%%%%%%%%%%%%%%%%%%

\section{3-D Riordan Arrays and Periodic Orbits}

The three-dimensional (or 3-D) Riordan array
\begin{equation}
\label{3DRA1} 
\bigl(g(t),\,f(t),\,h(t)\bigr), ~ \mbox{where} ~ g(t),\,h(t)\in {\cal F}_0,~ f(t) \in {\cal F}_1
\end{equation}
is a generalization of the (usual or 2-D) Riordan array $\bigl(g(t),\,f(t)\bigr)$ via an additional 
f.p.s. $h(t)$, and consists of the infinitely many 2-D Riordan arrays 
\begin{equation}
\label{3DRAs} 
\bigl(g(t)h^s(t),\,f(t)\bigr)
\end{equation}
indexed by non-negative integers $s$.
Such Riordan arrays can be described as 3-D matrices, 
${\Theta} = [\gamma_{i,j,s}]_{i,j,k\geq 0} =\bigl(g(t),\,f(t),\,h(t)\bigr)$, with
$$
\gamma_{i,j,s} = [t^i]gf^jh^s, ~ i,j,s\geq 0
$$
representing the entry of $\Theta$ in the $i$-th row, $j$-th column, and $s$-th layer. 
In particular, for each fixed $s$, the $s$-th layer $\L_s(\Theta)$ of $\Theta$ 
is a proper 2-D Riordan array given by two f.p.s., $g(t)h^s(t)$ and $f(t)$. 
That is, using the standard notation for the Riordan arrays, 
$$
\L_s(\Theta) = (gh^s,f).
$$

In Figure \ref{3DPic1}, we show the first few rows, columns, and layers of the 3-D
Riordan array $\left(\frac{4+t}{1+t+t^2},\,t(4+2t+2t^2),\,1 + 3t\right)$ over $\F_{11}$.

{\small
\begin{minipage}{\textwidth}
\centering
\usetikzlibrary{matrix,calc}
\begin{tikzpicture}[every node/.style={anchor=north east,
fill=white,minimum width=0.4cm,minimum height=4mm}]
\matrix (mD) [draw,matrix of math nodes,opacity=0]
{
A & 0 & 0 & 0 & 0 & 0 &\dots \\
0 & 5 & 0 & 0 & 0 & 0 &\dots \\
3 & 8 & 9 & 0 & 0 & 0 &\dots \\
0 & 9 & 9 & 3 & 0 & 0 &\dots \\
2 & 6 & 7 & 10 & 1 & 0 &\dots \\
9 & 3 & 3 & 9 & 2 & 4 &\dots \\
\vdots & \vdots & \vdots & \vdots & \vdots &  \vdots & \ddots \\
};
\matrix (mC) [draw,matrix of math nodes] at ($(mD.south west)+(3.6,3.6)$)
{
4 & 0 & 0 & 0 & 0 & 0 &\dots \\
10 & 5 & 0 & 0 & 0 & 0 &\dots \\
6 & 4 & 9 & 0 & 0 & 0 &\dots \\
4 & 8 & 4 & 3 & 0 & 0 &\dots \\
1 & 4 & 6 & 1 & 1 & 0 &\dots \\
6 & 2 & 7 & 6 & 10 & 4 &\dots \\
\vdots & \vdots & \vdots & \vdots & \vdots &  \vdots & \ddots \\
};
\matrix (mB) [draw,matrix of math nodes,opacity=1] at ($(mC.south west)+(0.9,0.8)$)
{
4 & 0 & 0 & 0 & 0 & 0 &\dots \\
9 & 5 & 0 & 0 & 0 & 0 &\dots \\
1 & 0 & 9 & 0 & 0 & 0 &\dots \\
1 & 8 & 10 & 3 & 0 & 0 &\dots \\
9 & 2 & 9 & 3 & 1 & 0 &\dots \\
1 & 7 & 2 & 8 & 7 & 4 &\dots \\
\vdots & \vdots & \vdots & \vdots & \vdots &  \vdots & \ddots \\
};
\matrix (mA) [draw,matrix of math nodes,opacity=1] at ($(mB.south west)+(1,.8)$)
{
4 & 0 & 0 & 0 & 0 & 0 &\dots \\
8 & 5 & 0 & 0 & 0 & 0 &\dots \\
10 & 7 & 9 & 0 & 0 & 0 &\dots \\
4 & 9 & 5 & 3 & 0 & 0 &\dots \\
8 & 8 & 5 & 5 & 1 & 0 &\dots \\
10 & 5 & 9 & 4 & 4 & 4 &\dots \\
\vdots & \vdots & \vdots & \vdots & \vdots &  \vdots & \ddots \\
};
\draw[dashed](mA.north east)--(mD.north east);
\draw[thick,-stealth](mA.north west)-- node[sloped,above] 
{Layers ($s \geq 0$)} (mD.north west);
\draw[dashed,](mA.south east)--(mD.south east);
\draw[thick,-stealth] (mA.north west)
   -- (mA.north east) node[midway,above] {Cols. ($j \geq 0$)};
\draw[thick,-stealth] (mA.north west)
   -- (mA.south west) node[midway,below,rotate=270] {Rows ($i \geq 0$)};

\draw (mA-7-7) to[out=260,in=140] ++(.4cm,-1cm) node[below] {Layer $s=0$};

\draw (mB-7-7) to[out=260,in=140] ++(.4cm,-1cm) node[below] {Layer $s=1$};

\draw (mC-7-7) to[out=260,in=140] ++(.4cm,-1cm) node[below] {Layer $s=2$};
\end{tikzpicture}
\captionof{figure}{3-D Riordan array 
$\left(\frac{4+t}{1+t+t^2},\,t(4+2t+2t^2),\,1 + 3t\right)$ over $\F_{11}$}
\label{3DPic1}
\end{minipage}
}

~

Such three-dimensional arrays were introduced by Cheon and Jin in 2017, 
(see \cite{Cheon}, and also \cite[Chapter~7]{Shapiro0}) as elements of the 
three-dimensional Riordan group. Here, we will use a 3-D Riordan array as 
an infinite sequence of layers, where each layer is a certain 2-D Riordan 
array determined by four polynomials with coefficients in $\F_q$.

Thus, we take four polynomials $p_i(t)\in \F_q[t], ~ i\in\{1,\,2,\,3,\, 4\}$, where $p_i(0)\neq 0$, 
$d_i = \deg(p_i(t))$, $d_2\leq d_1 + 1$, and $\gcd(p_2(t),\,p_j(t))=1$ for each $j\in\{1,\,3,\,4\}$, 
and consider the 3-D Riordan array $\Theta = \bigl(p_1(t)/p_2(t),\,tp_3(t),\,p_4(t)\bigr)$. 
For an integer $s\geq 0$, the $k$-th column of the $s$-th layer  
\begin{equation}
\label{3DRAp2}
\L_s(\Theta) = \left(\frac{p_1(t)p^s_4(t)}{p_2(t)},\,tp_3(t)\right)
\end{equation}
is given by the coefficients of the f.p.s. $\Bigl(p_1(t)p_4^s(t)(tp_3(t))^k\Bigr)/p_2(t)$, and as we 
know from Theorem \ref{MainThm1}, the periodic block in this column equals
$$
C_k = A^k\cdot C_0 = {\rm circ}^k(\vec{v}_3)\cdot C_0.
$$ 
As above, matrix $A = {\rm circ}(\vec{v}_3)$ is generated by the coefficients of $p_3(t)$ according 
to (\ref{Vec1}), or (\ref{Vec2}), and $C_0$ is the periodic block in the 0-th column determined by the 
coefficients of the f.p.s. $p_1(t)p_4^s(t)/p_2(t)$. 

Since the set of powers of ${\rm circ}(\vec{v}_3)$ over $\F_q$ is finite, we have 
only finitely many distinct periodic blocks. For example, for the Riordan array
\begin{equation}
\label{RAExamp1}
\left(\frac{1+t+t^4}{1+t^3},\,t(1 + t^2)\right) \in \F_2[[t]]\times \F_2[[t]]
\end{equation}
the reader will easily check that there are only four different periodic blocks, starting with  
the initial block $C_0 = (0,\,1,\,0)^T$. These blocks are
$$
C_1 = C_{1 + 3k} = \begin{pmatrix}
0\\
1\\
1\\
\end{pmatrix}, ~ C_2 = C_{2 + 3k} = \begin{pmatrix}
1\\
1\\
0\\
\end{pmatrix}, ~ C_3 = C_{3 + 3k} = \begin{pmatrix}
1\\
0\\
1\\
\end{pmatrix},~k\in\natu.
$$
In other words, the orbit of $C_0$ under the matrix ${\rm circ}(\vec{v}_3)$ is finite. 
Since such a set of periodic blocks is completely determined by the Riordan array 
$\bigl(p_1(t)/p_2(t),\,tp_3(t)\bigr)$, we will call it the {\sl orbit in $\bigl(p_1(t)/p_2(t),\,tp_3(t)\bigr)$}. 

\begin{definition}
Consider the Riordan array $\L =\bigl(p_1(t)/p_2(t),\,tp_3(t)\bigr)$, where the polynomials $p_i(t)$ 
satisfy all requirements from Theorem \ref{MainThm1}. Then, we define {\sl the orbit in $\L$} 
to be the set
\begin{equation}
\label{Orbit1}
\O(\L) = \left\{ {\rm circ}(\vec{v}_3)^k\cdot C_0 ~|~k\in\natu_0 \right\},
\end{equation}
where $C_0$ is the periodic block in the zeroth column of $\L$.
\end{definition}

\noindent For example, the orbit in the Riordan array (\ref{RAExamp1}) is the set
$$
\left\{
\begin{pmatrix}
0\\
1\\
0\\
\end{pmatrix}, ~
\begin{pmatrix}
0\\
1\\
1\\
\end{pmatrix}, ~ \begin{pmatrix}
1\\
1\\
0\\
\end{pmatrix}, ~ \begin{pmatrix}
1\\
0\\
1\\
\end{pmatrix}
\right\}.
$$

When the periodic blocks have length $\leq 3$ (i.e. the period $\leq 3$), 
we can visualize the orbits by plotting the blocks as points in the corresponding 
$n$-space, and connecting consecutive points by a line segment. 
In Figure \ref{Orbit1A}, we present such a visualization for the above example of  
the Riordan array (\ref{RAExamp1}). 

\begin{figure}[h] %  figure placement: here, top, bottom, or page
\centering
\includegraphics[width=60mm]{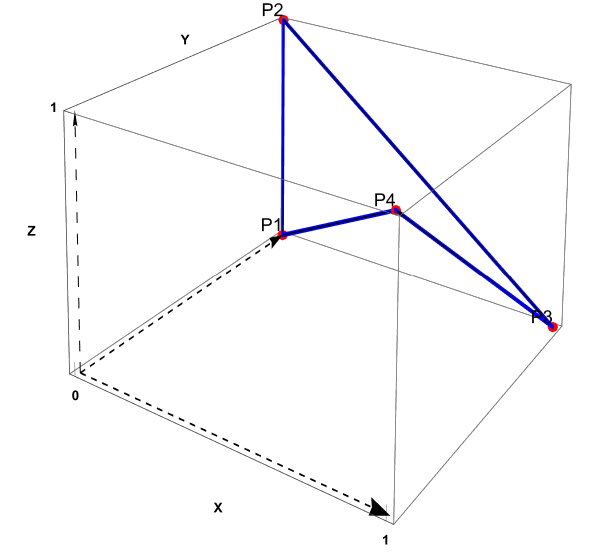} 
\caption{Orbit in the Riordan array (\ref{RAExamp1})}
\label{Orbit1A}
\end{figure}

In Figure \ref{Orbit2C}, we show a more interesting
example of an orbit in the Riordan array 
\begin{equation}
\label{RAExamp2}
\L = \left(\frac{(4+t)}{1+t+t^2},\, t(4+2t+2t^2)\right) 
\in \mathbb{F}_{11}[[t]]\times \mathbb{F}_{11}[[t]].
\end{equation}
The periodic blocks in this array are of length $3$. This 2-D Riordan 
array is the 0th layer of the 3-D Riordan array in Figure 1,
\begin{equation}
\label{Orbit2B}
\mathcal{O}(\L) = \left\{ \operatorname{circ}^k(\vec{v}_3) 
\cdot C_0\;\middle|\; k \in \mathbb{N}_0 \right\},
\end{equation}
where
$$
{circ}(\vec{v}_3)=A=
\begin{pmatrix}
2 & 2 & 4 \\
4 & 2 & 2 \\
2 & 4 & 2
\end{pmatrix},
$$
and
$$
C_0 =
\begin{pmatrix}
4 \\
8 \\
10
\end{pmatrix}.
$$
Matrix $A$ has order $30$ modulo 11, and the orbit consists of $30$ distinct points. 

\begin{figure}[h] %  figure placement: here, top, bottom, or page
\centering
\includegraphics[width=60mm]{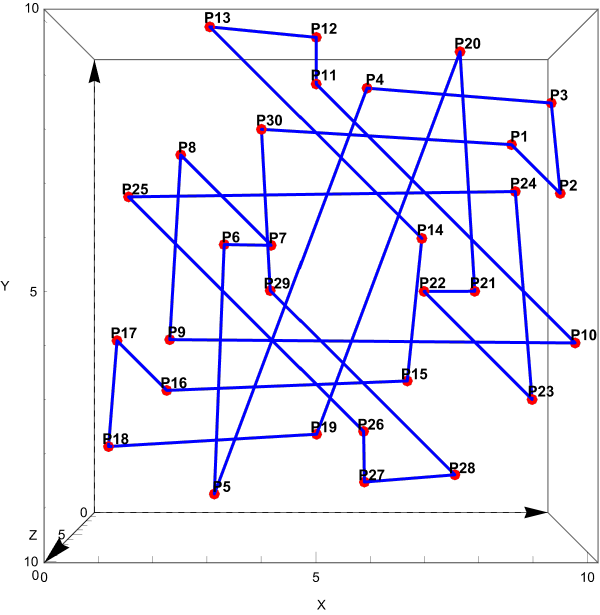} 
\caption{Orbit in the Riordan array (\ref{RAExamp2})}
\label{Orbit2C}
\end{figure}

We observed that the orbits in the layers of a 3-D Riordan array of polynomials 
also exhibit periodic behavior. In the next theorem, we show that such 
periodicity in the 3-D Riordan array $\Theta$ is determined by a circulant 
matrix generated from the coefficients of $p_4(t)$.

\begin{theorem}
\label{MainThm2}

Take four polynomials $p_i(t)\in \F_q[t], ~ i\in\{1,\,2,\,3,\, 4\}$, where $p_i(0)\neq 0$, 
$d_i = \deg(p_i(t))$, $d_2\leq d_1 + 1$, and $\gcd(p_2(t),\,p_j(t))=1$ for each $j\in\{1,\,3,\,4\}$, 
and consider the 3-D Riordan array $\Theta = \bigl(p_1(t)/p_2(t),\,tp_3(t),\,p_4(t)\bigr)$.
Let $\pi$ be the least period of the L.R.S. determined by 
the f.p.s. $1/p_2(t)$. Further, let $p_4(t) = z_0 + z_1t + \cdots + z_{d_4}t^{d_4}$, 
and $\M_4=\c(\vec{v}) \in \rm{GL}(\pi,\,\F_q)$ be the circulant $\pi\times\pi$ 
matrix generated by the vector

\begin{equation}
\label{Vec1b}
\vec{v} = \Bigl( \sum\limits_{i=0}^{\lfloor\frac{d_4}{\pi}\rfloor} z_{d_4 - \pi i},\, 
\sum\limits_{i=0}^{\lfloor\frac{d_4 - 1}{\pi}\rfloor} z_{d_4 - 1 - \pi i},\,\ldots ,\, 
\sum\limits_{i=0}^{\lfloor\frac{d_4 -  (\pi -1)}{\pi}\rfloor} 
z_{d_4 - (\pi - 1) - \pi i} \Bigr) 
\end{equation}
if $\pi < d_4 + 1$, or by the vector
\begin{equation}
\label{Vec2b}
\vec{v} = (z_{d_4},\, z_{d_4 - 1} ,\, \ldots, \, z_1 ,\, z_0,\, 0,\,\ldots,\, 0) ~ 
\mbox{if} ~ \pi \geq  d_4 + 1.
\end{equation}
Then, for each $k\geq 0$, the orbit $\O(\L_k)$ in the $k$-th layer of $\Theta$  
satisfies the pointwise relation  
\begin{equation}
\label{State2}
\O(\L_k) = \M_4^k \cdot \O(\L_0),
\end{equation}
where $\O(\L_0)$ is the orbit in the 
Riordan array $\L_0 = \bigl(p_1(t)/p_2(t),\,tp_3(t)\bigr)$.
\end{theorem}
%%%%%%%%   Proof of the main theorem    %%%%%%%%%%%%
\begin{proof}
Let $A$ denote the circulant matrix $\c(\vec{v})$ generated by the coefficients of $p_3(t)$, 
as in Theorem \ref{MainThm1}, and $C_{n,k}$ denote the periodic block in the 
$n$-th column of the $k$-th layer of $\Theta$. Suppose for a moment that the periodic 
block in the 0th column of the $k$-th layer is obtained from the corresponding block in the 
0th layer of $\Theta$, according to $C_{0,k} = \M_4^k\cdot C_{0,0}$. Then, using statement (ii) of 
Theorem \ref{MainThm1}, along with the commutativity of multiplication of circulant 
matrices, we have
$$
%\begin{equation}
%\label{CircCom1}
\M_4^k \cdot C_{n,0} = \M_4^k \cdot A^n\cdot C_{0,0} =   A^n \cdot \M_4^k \cdot C_{0,0}
= A^n\cdot C_{0,k} = C_{n,k}.
%\end{equation}
$$
Therefore, the orbit in the 0th layer is transformed into the 
orbit in the $k$-th layer via the matrix $\M_4^k$, preserving the order of the 
points (i.e., the columnwise order of the periodic blocks). In particular, it proves (\ref{State2}). 
Hence, to complete the proof, we must show that $C_{0,k} = \M_4^k\cdot C_{0,0}$. 
We do so for the case when $\pi \geq d_4 + 1$. The other case can be finished in the same 
way as in the analogous case in Theorem \ref{MainThm1}. 

According to the lemmas from Section 2, as well as the requirements on the polynomials 
$p_1(t),\,p_2(t)$, and $p_4(t)$, we have for any $k\geq 0$, 
$$
C_{0,k} = \begin{pmatrix}
[t^{(kd_4 + d_1 + 1) - d_2}]\left(\frac{p_1(t)p^k_4(t)}{p_2(t)}\right)\\
[t^{(kd_4 + d_1 + 1) - d_2 +1}]\left(\frac{p_1(t)p^k_4(t)}{p_2(t)}\right)\\ \vdots \\ 
[t^{(kd_4 + d_1 + 1) - d_2 + 1 +\pi-1}]\left(\frac{p_1(t)p^k_4(t)}{p_2(t)}\right)
\end{pmatrix},~\mbox{where} ~ 1 + d_1 - d_2 \geq 0.
$$

Let us use the induction on $k$, and since the base step when $k = 0$ is trivially true, 
we assume next that $k\geq 1$, $C_{0,k-1} = \M_4^{k-1}\cdot C_{0,0}$, and aim to show   
$C_{0,k} = \M_4\cdot C_{0,k-1}$. 
%Following notations from Theorem \ref{MainThm1}, 
%we denote the periodic block $C_{0,k-1}$ by 
%$$
%C_{0,k-1} = (s_r,\,s_{r+1},\, \ldots,\, s_{r + \pi -1})^T,~ 
%\mbox{where} ~ r = (k-1) d_4 + d_1 + 1 - d_2 \geq 0.
%$$
Writing the ratio 
$$
\frac{p_1(t)p_4^{k}(t)}{p_2(t)} = 
(z_0 + z_1t + \cdots + z_{d_4}t^{d_4})\frac{p_1(t)p_4^{k-1}(t)}{p_2(t)}
$$
as
%\begin{equation}
%\label{Rat22}
$$
\frac{p_1(t)p_4^{k}(t)}{p_2(t)} = 
(z_0 + z_1t + \cdots + z_{d_4}t^{d_4})\sum\limits_{i=0}^{\infty} s_it^i,
%\end{equation}
$$
and applying the properties of the operator $[t^n]$, we deduce that 
for each $m\in\{0,\ldots, \pi-1\}$, 
$$
[t^{(kd_4 + d_1 + 1) - d_2 + m}] \frac{p_1(t)p_4^{k}(t)}{p_2(t)}
$$
\begin{equation}
\label{Rat23}
= \sum\limits_{j = 0}^{d_4}z_j[t^{((k-1)d_4 + d_1 + 1) - d_2 + m - j}]\frac{p_1(t)p_4^{k-1}(t)}{p_2(t)}
\end{equation}
(recall that $\pi \geq d_4 + 1$). Using the shift operator $T$ from the definition of a 
circulant matrix, we see that equality (\ref{Rat23}) implies that the $m$-th term of 
the block $C_{0,k}$ equals the dot product of the vector
$$
T^m(z_{d_4},\,z_{d_4 -1},\,\ldots,\, z_{0},\,0,\,\ldots,\,0)
$$
with the column vector $C_{0,k-1}$. This proves that $C_{0,k} = \M_4\cdot C_{0,k-1}$.
\end{proof}

To illustrate this theorem, we show in Figure \ref{Orbit3} 
graphs of the orbits $\O(\L_0)$ and $\O(\L_1) = \M_4\cdot \O(\L_0)$ 
in the 3-D Riordan array from Figure 1. 
 
\begin{figure}[h] %  figure placement: here, top, bottom, or page
\centering
\includegraphics[width=63mm]{PicSF.pdf} ~~ \includegraphics[width=63mm]{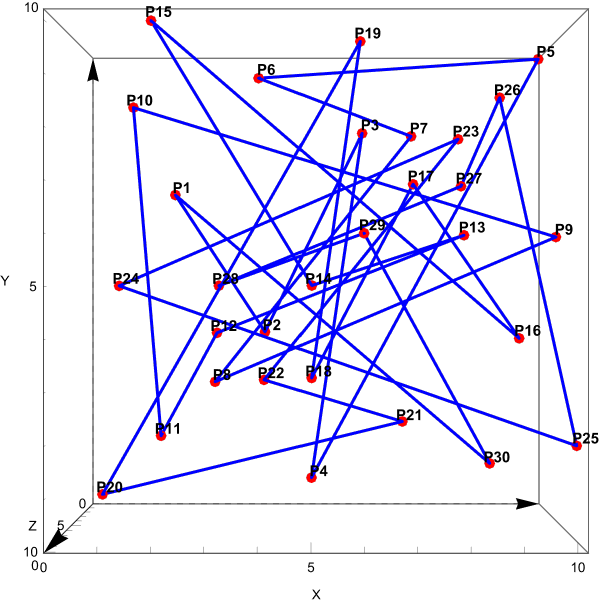} 
\caption{Two orbits in the layers of the 3-D Riordan array from Figure 1.}
\label{Orbit3}
\end{figure}

%%%%%%%%%%%%%%%%%%%%%%%%%%%
%%%%%%%%%%%%%%%%%%%%%%%%%%%
%%%%%%%%                      %%%%%%%%%%%%
%%%%%%%%   Section 5    %%%%%%%%%%%%
%%%%%%%%                      %%%%%%%%%%%%
%%%%%%%%%%%%%%%%%%%%%%%%%%%
%%%%%%%%%%%%%%%%%%%%%%%%%%%

\section{Preperiodic Column Partial Sums}

Let us begin this section with an example, and consider the Riordan array
\begin{equation}
\label{StExampS5}
\left(\frac{1}{1+t^2},\,t(1+t^3)\right) \in \F_3[[t]]\times \F_3[[t]].
\end{equation}
The corresponding coefficient matrix is as follows.
$$
\left(
\begin{BMAT}[3pt]{cccccccccc}{ccccccccccccc}
1 & 0 & 0 & 0 & 0 & 0 & 0 & 0 & 0  & \cdots \\
0 & 1 & 0 & 0 & 0 & 0 & 0 & 0 & 0  & \cdots \\
2 & 0 & 1 & 0 & 0 & 0 & 0 & 0 & 0  & \cdots \\
0 & 2 & 0 & 1 & 0 & 0 & 0 & 0 & 0  & \cdots \\
1 & 1 & 2 & 0 & 1 & 0 & 0 & 0 & 0  & \cdots \\
0 & 1 & 2 & 2 & 0 & 1 & 0 & 0 & 0  & \cdots \\
2 & 2 & 1 & 0 & 2 & 0 & 1 & 0 & 0  & \cdots \\
0 & 2 & 1 & 1 & 1 & 2 & 0 & 1 & 0  & \cdots \\
1 & 1 & 0 & 0 & 1 & 2 & 2 & 0 & 1  & \cdots \\
0 & 1 & 2 & 2 & 2 & 1 & 0 & 2 & 0  & \cdots \\
2 & 2 & 0 & 0 & 2 & 1 & 1 & 1 & 2  & \cdots \\
0 & 2 & 1 & 1 & 1 & 0 & 0 & 1 & 2  & \cdots \\
\vdots & \vdots & \vdots & \vdots & \vdots & \vdots & \vdots & \vdots & \vdots & \ddots
\addpath{(0,9,1)ruuuuldddd}
\addpath{(1,6,1)ruuuuldddd}
\addpath{(2,2,1)ruuuuldddd}
\end{BMAT} \right)
$$
The periodic block in the $k$-th column ($k\geq 0$) starts with the 
$(4k - 1)$-st term. If we add up all of the $4k-1$ terms above, we 
will see that the sequence of such partial sums modulo 3 will be also periodic, 
with the periodic block $\{0,1,0,0,1,1,1,2\}$. 

Periodic behavior of such partial sums in the columns of certain Riordan arrays 
was first noticed and studied by the third author in \cite{Krylov1} and \cite{Krylov2}. 
For example, considering over the reals the Riordan array $\bigl(1/(1 - t^3),\,tp(t)\bigr)$ 
with $p(t) = (2 - t + 2t^2)/3$, one notices that the sequence of such partial sums 
is not periodic, but for $p(t) = (-1 + 2t + 2t^2)/3$ it will be periodic with period 6 
(see \cite{Krylov1}, Tables 4 and 5). When we consider the Riordan arrays of 
polynomials over the finite fields, such sequences of preperiodic column sums 
will always be periodic. We prove it in Theorem \ref{PPSThm} below. 

Let us now define {\sl preperiodic column partial sums} 
of the Riordan array (\ref{RAp123}), recalling that the periodicity in the $k$-th 
column of (\ref{RAp123}) starts at the $\bigl((d_3 + 1)k+ d_1 - d_2 + 1\bigr)$-st term.

\begin{definition}
\label{DefPPS}
Let $p_i(t)\in\F_q[t],\,i\in\{1,\, 2,\,3\}$ be polynomials such that 
$p_i(0)\neq 0$, $\deg(p_i)=d_i$, and $p_2(t)$ is relatively prime to $p_1(t)$ 
and $p_3(t)$. Given the Riordan array $\left(\frac{p_1(t)}{p_2(t)},\, tp_3(t)\right)$, 
and the formal power series expansion 
$$
\frac{p_1(t)\bigl(tp_3(t)\bigr)^k}{p_2(t)} = \sum \limits_{i=0}^{\infty}s_{i,k} t^i,~k\geq 0,
$$ 
we define the {\sl $k$-th preperiodic partial sum} as  
\begin{equation}
\label{RACPS1}
S_{[k]} = \sum \limits_{i=0}^{(d_3 + 1)k+ d_1 - d_2} s_{i,k},
\end{equation}
when $(d_3 + 1)k+ d_1 - d_2 \geq 0$, and as $S_{[k]} = 0$ otherwise.
\end{definition}

Before stating the next Theorem \ref{PPSThm}, we notice that the column 
partial sums of any Riordan array coincide with the entries in the Riordan 
array, which is a product of $\bigl(1/(1 - t),\, t\bigr)$ with the original Riordan 
array. This is indeed clear, because the 
$n$-th partial sum $S_n$ of the $m$-th column can be presented as the dot
product of the infinite row vector $\{1,\ldots,1,0,0,\ldots\}$, which has 1 in the first 
$n$ positions and 0 after that, with the $m$-th column of the Riordan array 
(cf. Partial Sum Theorem 2.7 from \cite{Shapiro0}). Therefore, our sequence 
$\{S_{[k]}\}_{k\geq 0} ^{\infty}$ is a sequence of 
particular elements of the Riordan array 
\begin{equation}
\label{RAproduct1}
\left(\frac{1}{1-t},\, t\right) \cdot \left(\frac{p_1(t)}{p_2(t)},\, tp_3(t)\right) = 
\left(\frac{p_1(t)}{(1 - t)p_2(t)},\, tp_3(t)\right),
\end{equation}
and we can use the operator $[t^n]$ to obtain formulas 
for the elements of (\ref{RAproduct1}). 

\begin{theorem}
\label{PPSThm}
Consider the Riordan array 
$$
\left(\frac{p_1(t)}{p_2(t)},\, tp_3(t)\right) \in \F_q[[t]]\times \F_q[[t]],
$$
where each $p_i(t)\in\F_q[t]$ is a polynomial, such that 
$p_i(0)\neq 0$, $\deg(p_i)=d_i$, $d_2\leq d_1 + 1$, and $p_2(t)$ is 
relatively prime to $p_1(t)$ and $p_3(t)$. Then, the sequence of 
preperiodic partial sums $\{S_{[k]}\}_{k\geq 0}$ is (eventually) periodic.
\end{theorem}
\begin{proof}
Since for each $k\geq 0$ the sum $S_{[k]}$ adds the first 
$(d_3 + 1)k+ d_1 - d_2 + 1$ terms, it follows from (\ref{RAproduct1}) that for each $k\geq 1$, 
$$
S_{[k]} = [t^{(d_3 + 1)k+ d_1 - d_2 }]\frac{p_1(t)\bigl(tp_3(t)\bigr)^k}{(1-t)p_2(t)} 
= [t^{d_3k+ d_1 - d_2 }]\frac{p_1(t)\bigl(p_3(t)\bigr)^k}{(1-t)p_2(t)}
$$
\begin{equation}
\label{ParSumCoef}
= [t^{d_1 - d_2 }]\left(\frac{p_1(t)}{(1-t)p_2(t)}\right)
\left(\frac{p_3(t)}{t^{d_3}}\right)^k.
\end{equation}
As mentioned in the introduction, Theorem 8.40 of \cite{FF} implies 
that to prove eventual periodicity of the formal power series
\begin{equation}
\label{LFPS1}
S(x) = \sum\limits_{k=0}^{\infty} S_{[k]}x^k,
\end{equation}
it is enough to show that this series can be written as a 
rational function over $\F_q$, that is $S(x) \in \F_q(x)$. If we introduce  
$R_1(t) = p_1(t)/\bigl((1-t)p_2(t)\bigr)$ and $D=d_1 - d_2$ for brevity, then 
$$
S(x) = \sum\limits_{k=0}^{\infty} [t^D]R_1(t)\left(\frac{p_3(t)}{t^{d_3}}\right)^k x^k
$$
\begin{equation}
\label{LFPS2}
= [t^D]R_1(t) \sum\limits_{k=0}^{\infty} \left(\frac{p_3(t)}{t^{d_3}}\right)^k x^k 
=[t^D]\frac{t^{d_3}R_1[t]}{t^{d_3} - p_3(t)x}.
\end{equation}
The last fraction in (\ref{LFPS2}) is a rational function in two variables, and 
taking a fixed coefficient in $t$ from the f.p.s. expansion of this function will give a 
rational function in $x$. Indeed, suppose we have 
$$
H(t,\,x) = \frac{P(t,\,x)}{Q(t,\,x)},~\mbox{where}~ P,\,Q\in \F_q[t,\,x],
$$
and consider $[t^D]H(t,\,x)$. We can expand $H(t,\,x)$ into a formal Laurent series in $t$ 
over the field of rational functions in $x$, $\F_q(x)$. Since in our case 
$$
Q(t,\,x) =(1-t)p_2(t)\bigl(t^{d_3} - p_3(t)x\bigr),
$$
we have $Q(0,\,x) = -p_2(0)p_3(0)x \neq 0$. Therefore the ratio 
$1/Q(t,\,x)$ has a well defined f.p.s. expansion in $t$ with 
coefficients in the field of rational 
functions $\F_q(x)$, that is, we can write 
\begin{equation}
\label{LFPS3}
H(t,\,x) = P(t,\,x)\sum\limits_{j\geq 0}f_j(x)t^j,
\end{equation}
where $f_j(x)\in \F_q(x)$ for all $j\ge 0$. Since $P(t,\,x)$ is a polynomial, 
the convolution rule of $[t^n]$ applied to the product in (\ref{LFPS3}), 
implies that $[t^D]H(t,\,x)$ will be a finite sum of rational functions, and 
thus a rational function itself. 
\end{proof}

\begin{note}
The only place where we used the restriction $d_2\leq d_1 +1$, 
was the index of the first term in the column periodic blocks. 
The proof works equally well for any fixed integer $d_1 - d_2$, regardless of 
the inequality, provided that the index of the first term depends linearly on the 
column index $k$. Moreover, the GCD restriction seems to be also removable; 
however, to obtain precise formulas similar to \eqref{LFPS1} and \eqref{LFPS2}, 
one must first understand how exactly the starting term in each periodic block 
depends on the degrees of $p_i(t)$ (recall our Example \ref{Example5}).
\end{note}

Having established that the sequence of preperiodic column partial sums 
is eventually periodic, the next step is to characterize their periodic properties.
We close our study here with a particular family of Riordan arrays, 
for which such sequences are trivially constant. As a motivating example, 
consider the Riordan array over $\mathbb{F}_3$,
$$ 
\left( \frac{1}{1+t},\, t(1+t^4)\right),
$$
with the coefficient matrix
$$
\left(
\begin{BMAT}[3pt]{cccccccccc}{ccccccccccccc}
1 & 0 & 0 & 0 & 0 & 0 & 0 & 0 & 0  & \cdots \\
2 & 1 & 0 & 0 & 0 & 0 & 0 & 0 & 0  & \cdots \\
1 & 2 & 1 & 0 & 0 & 0 & 0 & 0 & 0  & \cdots \\
2 & 1 & 2 & 1 & 0 & 0 & 0 & 0 & 0  & \cdots \\
1 & 2 & 1 & 2 & 1 & 0 & 0 & 0 & 0  & \cdots \\
2 & 2 & 2 & 1 & 2 & 1 & 0 & 0 & 0  & \cdots \\
1 & 1 & 0 & 2 & 1 & 2 & 1 & 0 & 0  & \cdots \\
2 & 2 & 0 & 1 & 2 & 1 & 2 & 1 & 0  & \cdots \\
1 & 1 & 0 & 2 & 2 & 2 & 1 & 2 & 1  & \cdots \\
2 & 2 & 0 & 1 & 1 & 0 & 2 & 1 & 2  & \cdots \\
1 & 1 & 1 & 2 & 2 & 0 & 1 & 2 & 1  & \cdots \\
2 & 2 & 2 & 1 & 1 & 0 & 2 & 2 & 2  & \cdots \\
\vdots & \vdots & \vdots & \vdots & \vdots & \vdots & \vdots & \vdots & \vdots & \ddots
\addpath{(0,11,1)ruuldd}
\addpath{(1,6,1)ruuldd}
\addpath{(2,1,1)ruuldd}
\end{BMAT} \right).
$$

Based on the Definition \ref{DefPPS}, for the preperiodic partial 
sums here, we add the first $5k$ terms of $k$-th column. 
From the matrix, it is clear that the first $5k$ terms of $S_{[k]}$ for 
$k\in\{0,1,2\}$ are all equal to 0, and the same holds true for every 
$k\geq 0$. It will follow from our next result, where we 
generalize this example and show that all preperiodic 
column partial sums in the corresponding Riordan array 
over $\F_q$ vanish. 

\begin{corollary}
Take a finite field $\F_q$ of odd characteristic $p$, and 
a polynomial $p_3(t) = \sum\limits_{j = 0}^{d_3}c_jt^{2j}\in\F_q[t]$ of degree $2d_3$,  
such that $c_0\neq 0$ and $p\nmid \sum\limits_{j=0}^{d_3} c_j$.
Then the preperiodic partial sums in the columns of the Riordan array 
\begin{equation}
\label{LastRA1}
\left(\frac{1}{1+t},\,tp_3(t)\right)
\end{equation}
equal 0 modulo $p$, that is, $S_{[k]}\equiv 0\pmod{p}$ for all $k\geq 0$.
\end{corollary}
\begin{proof}
First, we note that for any positive integer $d_3$, the polynomials $p_3(t)$ 
and $1 + t$ are relatively prime over $\F_q$. It follows immediately from the 
observation that $p-1\equiv -1\pmod{p}$ is not a root of $p_3(t)$, since, 
according to the assumptions, $p$ 
does not divide the sum $c_0 + \cdots + c_{d_3}$. For the 
Riordan array (\ref{LastRA1}), 
we have $D = d_1 - d_2 = -1$, and \ref{ParSumCoef}) gives 
\begin{equation}
\label{LastPS2}
S_{[k]} = [t^{-1}]\left(\frac{1}{1 - t^2}\right)
\left(\frac{p_3(t)}{t^{2d_3}}\right)^k = 
[t^{-1}]\left(\frac{p_3(t)}{t^{2d_3}}\right)^k\sum\limits_{j\geq 0}t^{2j}.
\end{equation}
Since $p_3(t)$ has no odd-degree terms, the coefficient of $t^{-1}$ in 
the Laurent formal power series expansion in (\ref{LastPS2}) must be zero, 
which completes the proof. 
\end{proof}

\end{document}